\documentclass[12pt]{article}
\usepackage{amsthm}
\usepackage{amsmath}
\usepackage{amssymb}
\usepackage[margin=1in]{geometry}
\usepackage{hyperref}
\usepackage{xcolor}
\usepackage{tikz}
\hypersetup{
    colorlinks,
    linkcolor={black},
    citecolor={black},
    urlcolor={black}
}
\usepackage{setspace}
\usetikzlibrary{patterns}
\newcommand{\no}{\noindent}

\usetikzlibrary{matrix}

\newtheorem{theo}{Theorem}[section]

\newtheorem{lemma}[theo]{Lemma}
\newtheorem{defi}[theo]{Definition}
\newtheorem{coro}[theo]{Corollary}
\newtheorem{conj}[theo]{Conjecture}
\newtheorem{claim}[theo]{Claim}
\newtheorem{quest}[theo]{Question}

\newcommand{\QQ}{\mathbb{Q}} 
\newcommand{\NN}{\mathbb{N}}

\newcommand{\FF}{{\mathcal{F}}}  
\newcommand{\UU}{{\mathcal{U}}}  
 
\newcommand{\PP}{{\mathbb P}}

\newcommand{\SSS}{{\mathcal{S}}}  
\newcommand{\bs}{\,\backslash\,}

\begin{document}
\author{Jeremy Chizewer\thanks{
Department of Combinatorics and Optimization, University of Waterloo,
Waterloo, ON, CA, N2L 3G1. Department of Computer Science, Princeton University,
Princeton, NJ 08544, and Research supported in part by the Fifty Five Fund for Senior Thesis 
Research (Class 1955) Fund, Princeton University.
Email: {\tt jchizewer@uwaterloo.ca}.}
}

\title{On Restricted Intersections and the Sunflower Problem}
\date{}
\maketitle

\begin{abstract}
A sunflower with $r$ petals is a collection of $r$ sets over a ground set $X$ such that every 
element in $X$ is in no set, every set, or exactly one set. Erd\H{o}s and Rado~\cite{er} showed that a 
family of sets of size $n$ contains a sunflower if there are more than $n!(r-1)^n$ sets in the family. 
Alweiss et al.~\cite{alwz} and subsequently Rao~\cite{rao} and Bell et al.~\cite{bcw} improved this bound to $(O(r \log(n))^n$. 

We study the case where the pairwise intersections of the set family are restricted. In particular, we 
improve the best-known bound for set families when the size of the pairwise intersections of any 
two sets is in a set $L$.  We also present a new bound for the special case when the set $L$ is the 
nonnegative integers less than or equal to $d$ using the techniques of Alweiss et al.~\cite{alwz}.
\end{abstract}


\section{Introduction}

A \emph{set family} $\FF$ over a finite set $X$ is a collection of subsets of $X$.
We say a set family is \emph{$n$-uniform} if every set in the family has size $n$. 
\begin{defi}[Sunflower] An $r$-sunflower is a collection of sets $S_1,\dots, S_r$ such
that \[S_i\cap S_j=S_1\cap S_2\cap\cdots\cap S_r=K \text{ for all $i\neq j$.}\]
We call the set $K$ the \emph{core} and the sets $S_i\backslash K$ the \emph{petals.}
\end{defi}

Erd\H{o}s and Rado~\cite{er} originally referred to sunflowers as $\Delta$-systems 
and proved that given an $n$-uniform set family $\mathcal F$ with 
$|\mathcal {F}| =n!(r-1)^n$ there exists an $r$-sunflower contained in $\mathcal F$.
Sunflowers were renamed by Deza and Frankl in~\cite{df},
and the term sunflower is now more popular. The sunflower problem has been studied in
several papers including \cite{er, deza, df, alwz, rao, gh}.

Erd\H{os} and Rado conjectured that a stronger bound holds than the one they proved in the 
initial paper.
\begin{conj}[Erd\H{o}s and Rado~\cite{er}]\label{erdos-conj}
Let $\FF$ be an $n$-uniform set family. There exists some constant $C=C(r)$ depending
only on $r$ such that $\FF$ contains an $r$-sunflower whenever 
\[|\FF|> C^{\, n}.\]
\end{conj}
Recently, Alweiss et al.~\cite{alwz}, and subsequently Rao~\cite{rao} and Bell et al.~\cite{bcw}, made progress toward this 
bound by showing there exists some constant $C$ such that  $\FF$ contains an $r$-sunflower 
whenever $|\FF|> [Cr\log(n)]^n$. We study the Sunflower Problem with added restrictions on the 
pairwise intersections.

\begin{defi}
Let $\FF$ be a set family. We call $\FF$ an \emph{$L$-intersecting family}, 
if there exists a set $L\subset\NN$ such that 
\[|F_i\cap F_j|\in L \textrm{ for every } F_i,F_j\in \FF,\,\textrm{with } i\neq j.\]
\end{defi}

The problem of sunflowers in $L$-intersecting set families was first studied in~\cite{gh}. 
They show that given an $L$-intersecting set family $\FF$ with $|L|=s$, the family $\FF$
contains a $3$-sunflower whenever $|\FF|> (n^2-n+1)8^{s-1}2^{(1+\sqrt{5}/5)n(s-1)}$. 
We improve their bound, and extend the result to all 
$r\geq 3$ in the following theorem, which is one of our main results.

\begin{theo}\label{l-int}
Let $\FF$ be an $L$-intersecting, $n$-uniform set family, for some set $L\subset\NN$ with 
${|L|=s\geq 1}$. Let $m = \max\{r-1, n^2 - n+1\}$. Then $\FF$ contains an $r$-sunflower whenever
\[|\FF|> 2^{n\log_2(s+1) + s\log_2 (m)}.\]
\end{theo}

We also consider the special case where $L = \{0,1,\dots,d\}$ for some $d\in\NN$. In this case, we 
call the set family  \emph{$d$-intersecting}. Using the techniques of \cite{alwz,rao} we achieve the
following bound, our second main result.

\begin{theo}\label{intersecttheo} Let $\FF$ be a $d$-intersecting, $n$-uniform set family. 
There exists an absolute constant $C$ such that for every $r,n\geq 3$, the family 
$\FF$ contains an $r$-sunflower whenever
\[|\FF|> (4r)^n[Cr\log (rd)]^d.\]
\end{theo}

\begin{coro}\label{intcoro}
For any $C>1$, there exists $c=c(r)>0$ depending only on $C$ and $r$ such that if 
$\FF$ is an $n$-uniform, $r$-sunflower free family with 
$|\FF|\geq (C4r)^n$ for $3\leq r\leq \log n$ then there exists $F_1,F_2\in\FF$ such that 
\[|F_1\cap F_2| \geq c n/\log\log n.\]
\end{coro}
\no Corollary~\ref{intcoro} follows immediately from Theorem~\ref{intersecttheo} by setting $d=c n/\log\log n$.
The following question naturally arises.

\begin{quest}\label{lin-int-quest}
Do there exist constants $C,c>0$ depending only on $r$ and possibly each other 
such that if $\FF$ is an $n$-uniform, $r$-sunflower free family of size $|\FF|\geq C^{\, n}$, then 
there exist $F_1,F_2\in\FF$ such that $|F_1\cap F_2|\geq cn$?
\end{quest}

Corollary~\ref{intcoro} is of interest because the statement in Question~\ref{lin-int-quest} 
is equivalent to Conjecture~\ref{erdos-conj}. 

The remainder of the paper is organized as follows. Theorem~\ref{l-int} will be proved in 
Section~\ref{l-int-section}, and Theorem~\ref{intersecttheo} in Section~\ref{intersect}. 
We briefly discuss the results Section~\ref{discussion}.


\section{Proof of Theorem~\ref{l-int}}\label{l-int-section}

We proceed using similar ideas to the original proof of Erd\H{o}s and Rado~\cite{er}, 
beginning with the following lemma. 
\begin{lemma}[Deza~\cite{deza}]\label{deza-lemma}
Let $\FF$ be an $L$-intersecting, $n$-uniform set family where $L=\{t\}$ for some $0\leq t<n$. 
If $\FF\geq n^2-n+2$ then $\FF$ is a sunflower.\end{lemma}

Lemma~\ref{deza-lemma} provides the base case for our induction argument. We also state the
following definition, which will be useful in this section and the next one.

\begin{defi}
Given a family $\FF$ over $X$, and a set $T\subseteq X$ the ``link'' of $\FF$ at $T$,
denoted $\FF_T$, is defined as
\[\FF_T=\{F\bs T: F\in\FF, T\subseteq F\}.\]
\end{defi}

\begin{proof}[Proof of Theorem~\ref{l-int}]
We prove using induction a slightly stronger statement than that of the theorem. 
\begin{claim}\label{l-int-claim} Let $\FF$ be an $L$-intersecting, $n$-uniform set family for 
$L=\{\ell_1,\dots,\ell_s\}$ with ${0\leq \ell_1<\ell_2<\cdots<\ell_s<n}$ 
and $s\geq 1$. Let $m = \max\{r-1, n^2 - n+1\}$. Then $\FF$ contains an $r$-sunflower whenever 
\begin{align}\label{f-size}
|\FF|> \frac{n!m^{s}}{(\ell_1+1)!(\ell_2-\ell_1)!(\ell_3-\ell_2)!
\cdots(\ell_s-\ell_{s-1})!(n-\ell_s-1)!}.\end{align}
\end{claim}

\begin{proof}[Proof of claim]
Fix $m$ as above.
We proceed by induction on $s$. Indeed, for $s=1$, we apply Lemma~\ref{deza-lemma} to get 
the result immediately. Suppose the result holds for $0<j<s$, and let $\FF$ be an $n$-uniform,
$L$-intersecting family (with $L$ as above) with $|L|=s\geq 2$. 
Suppose that $\FF$ satisfies Inequality~(\ref{f-size}).
Let $\SSS\subseteq \FF$ be a maximal subset of $\FF$ such that for every 
$S_i,S_j\in \SSS$, if $i\neq j$, then $|S_i\cap S_j|=\ell_1$. If $|\SSS|> m$ then,
by Lemma~\ref{deza-lemma}, $\SSS$ (and hence $\FF$) contains an $r$-sunflower. Thus,
without loss of generality, we can assume that $|\SSS|\leq m$. By maximality of $\SSS$, 
every set in $\FF$ intersects at least one set of $\SSS$ in at least $\ell_1+1$ elements. Let 
$S\in \SSS$ be the set which intersects the most elements of $\FF$ in at least $\ell_1+1$ elements, 
and let \[\FF'=\{F\in\FF : |F\cap S|\geq \ell_1+1\}.\] 
By the pigeonhole principle $|\FF'|\geq |\FF|/m$. There are ${n\choose \ell_1+1}$ subsets of 
$S$ of size $\ell_1+1$, and every set in $\FF'$ contains
at least one such subset, so again by the pigeonhole principle there exists a set 
$S'\subseteq S$ such that $|S'|=\ell_1+1$ and the link at $S'$ in $\FF'$ satisfies 
\begin{align}|\FF'_{S'}|\geq \frac{|\FF|}{m{n\choose \ell_1+1}}
> \frac{(n-\ell_1-1)!m^{s-1}}{(\ell_2-\ell_1)!(\ell_3-\ell_2)!
\cdots(\ell_s-\ell_{s-1})!(n-\ell_s-1)!}\label{f-s-size}.\end{align}
Let $L'=\{\ell_2-\ell_1-1,\dots,\ell_s-\ell_1-1\}$. We observe that $\FF'_{S'}$ is an 
$(n-\ell_1-1)$-uniform, $L'$-intersecting family, and $|L'|=s-1$. Thus, 
by Inequality~(\ref{f-s-size}) and induction $\FF'_{S'}$ 
contains an $r$-sunflower (note that $m$ still satisfies the requirements of the induction hypothesis). 
Let $F_1,\dots,F_r\in \FF'_{S'}$ 
be an $r$-sunflower, then taking $F_1\cup S',\dots,F_r\cup S'\in \FF$ gives an $r$-sunflower.
\end{proof}
\no Theorem~\ref{l-int} follows immediately from Claim~\ref{l-int-claim} using the bound 
${n\choose m_1,\dots,m_k}\leq k^n$ for multinomial coefficients.
\end{proof}


\section{Proof of Theorem~\ref{intersecttheo}}\label{intersect}

We proceed using a similar argument to the main theorem of \cite{alwz}.
We start by stating some definitions from \cite{alwz}, so that we may apply the results. 

\begin{defi}
We say that an $n$-uniform family $\FF$ over $X$ is $\kappa$-spread if $|\FF|\geq \kappa^n$ 
and for all $T\subseteq X$ with $|T|\leq n$ we have $|\FF_T|\leq \kappa^{-|T|}|\FF|$
\end{defi}

We introduce weight functions, so that we can deal with multiset families (this is
not strictly necessary, but it makes it easier to apply the results of Alweiss et al. \cite{alwz}).

\begin{defi}
A function $\sigma: \FF\to \QQ$ is a weight function on a set family $\FF$ if it maps 
each set in $\FF$ to a rational weight, such that not all sets have weight zero.  
Moreover we define $\sigma(\SSS)=\sum_{S\in\SSS}\sigma(S)$ for a set family 
$\SSS\subseteq \FF$.
\end{defi}

Our next definition generalizes the idea of $\kappa$-spread using
weight functions.

\begin{defi}
We say that a set family $\FF$ over $X$, and corresponding weight function $(\FF,\sigma)$, is 
$\mathbf s$-spread if $\mathbf s=(s_0;s_1,\dots,s_n)$ satisfies 
$s_0\geq s_1\geq\cdots\geq s_n\geq0$ with $\sigma(\FF)\geq s_0$ and  
for every set $T\subseteq X$ the subfamily $\mathcal {T}=\{F\in \FF: T\subseteq F\}$, satisfies
$\sigma(\mathcal {T})\leq s_{|T|}$.
\end{defi}

Now we define $(\alpha,\beta)$-satisfying families, for the probabilistic
arguments that follow. We write $R\sim\UU(X,\alpha)$ whenever $R\subseteq X$
is generated by taking each element of $X$ uniformly and independently at 
random with probability $0\leq \alpha\leq 1$.

\begin{defi}\label{sat}
Let $0< \alpha,\beta< 1$. A family $\FF$ is $(\alpha,\beta)$-satisfying if given 
$R\sim \UU(X,\alpha)$,
\[\PP_R(\exists S\in \FF, S\subseteq R)> 1- \beta.\]
\end{defi}

Finally, we say that a family $\FF$ is \emph{$\mathbf s$-spread} if there exists a weight 
function $\sigma$ such that $(\FF,\sigma)$ is $\mathbf s$-spread, and 
a weight profile $\mathbf s$ is \emph{$(\alpha,\beta)$-satisfying} if any $\mathbf s$-spread
family $\FF$ is $(\alpha,\beta)$-satisfying. Using these definitions, we 
can now state the following lemmas. 

\begin{lemma}[\text{\cite[Lemma 1.6]{alwz}}]
\label{sat-disjoint}
If $\FF$ is a $(1/r,1/r)$-satisfying family, and $\emptyset\notin\FF$, then $\FF$ contains $r$ 
pairwise disjoint sets. 
\end{lemma}

\begin{lemma}[\text{\cite[Lemma 4]{rao}}]
\label{logn-spread-sat}
Let $0< \alpha,\beta< 1/2$. There exists a universal constant $C>1$ such that
if $\kappa=\kappa(n,\alpha,\beta) = C\log(n/\beta)/\alpha$ and a multiset family $\FF$ over 
$X$ is a $\kappa$-spread, $n$-uniform family then
$\FF$ is $(\alpha,\beta)$-satisfying.
\end{lemma}

The next lemma, which is the main technical result of this section, will allow us to use 
Lemma~\ref{logn-spread-sat} on sets of size $d$ for a $d$-intersecting family.

\begin{lemma}\label{good-size-d}
Let $\FF$ be a $d$-intersecting, $n$-uniform set family that is $\mathbf s$-spread,
such that $\mathbf s:=(|\FF|;s_1,\dots,s_d,1,\dots,1)$. Let $p,\delta>0$, and suppose that 
$\mathbf s'=((1-\delta)|\FF|;s_1,\dots,s_d)$ is $(\alpha',\beta')$-satisfying. Then
$\FF$ is $(\alpha,\beta)$-satisfying for 
\[\alpha = p + (1-p)\alpha' \textrm{, and \,} \beta=\beta'+(2/p)^n/(\delta |\FF|)\]
\end{lemma}

Before proving Lemma~\ref{good-size-d}, we define a notion of ``good'' and ``bad'' 
set pairs. Bounding the number of bad pairs is the key idea in this proof. 

\begin{defi}\label{goodbad}
Let $\FF$ be an $n$-uniform family over $X$ and let $W\subseteq X$. Given $S\in \FF$
and $w\in [n]$ we call the set pair $(W,S)_w$ good if there exists a set $S'\in \FF$ 
(possibly with $S'=S$) such that 
\[S'\bs W\subseteq S\bs W\text{, and } |S'\bs W|\leq w\]
We call $S'$ a \emph{witness} to the goodness of $(W,S)_w$.
We call a set bad otherwise. 
\end{defi}

We use Definition~\ref{goodbad} with $w = d$ for our purposes. 

\begin{proof}[Proof of Lemma~\ref{good-size-d}]
Let $\FF$ be a $d$-intersecting, $n$-uniform set family over $X$, with $|X|=x$. 
We begin by bounding the number of bad set pairs using an encoding inspired by~\cite{alwz}. 
Suppose $(W,S)_d$ is a bad pair for $W\subseteq X$ where $|W|=px$ and  $S\in \FF$.
First, we consider all possible sets $W\cup S$. Since $|S|=n$ and $|W|=px$, we know that
$px\leq |W\cup S|\leq px+n$. Hence, there are 
\[\sum_{i=0}^n{x\choose px+i} \leq \sum_{i=0}^n\left(\frac{1-p}{p}\right)^i {x\choose px} 
\leq \left(\frac{1-p}{p}+1\right)^{n} {x\choose px}  =  p^{-n}{x\choose px} \]
possible sets for $W\cup S$. Let $W\cup S$ be the first piece of information in the encoding.
There are $2^n$ possible values for $W\cap S$ since $|S|= n$. Let 
$W\cap S$ be the second piece of information in the encoding. Now we claim that 
given these two pieces of information and the additional information that the 
corresponding set pair $(W,S)_d$ is bad, we can reconstruct 
$(W,S)_d$. Indeed, if we knew $S$ in addition to this information, we could clearly reconstruct 
$(W,S)_d$. Let $S'\in \FF$ and suppose that
$S'\subseteq W\cup S$, so that $S'\bs W\subseteq S\bs W$.
Since $(W,S)_d$ is bad, it must be that $|S'\bs W|>d$, hence
$|S'\cap S|> d$. Since $\FF$ is $d$-intersecting, $S' = S$ and
there is a unique set $S$ for a given
$W\cup S$, which can be computed by taking the unique set $S\subseteq W\cup S$. 
Since we also know $W\cap S$, we can compute $W$. 
Therefore, there are at most $(2/p)^n{x\choose px}$ bad set
pairs. Since there are ${x\choose px}$ possible sets $W$, the expected
number of bad set pairs for a given $W$ is $(2/p)^n$. Let 
$\SSS(W)=\{S\in \FF: (W,S)_d\text{ is bad}\}$. 
By Markov's inequality, the probability over $W$ drawn uniformly from ${X\choose px}$, the
set of subsets of $X$ with size $px$, satisfies 
\begin{align}\PP_W(|\SSS(W)|\geq \delta |\FF|)\leq 
\frac{(2/p)^n}{\delta|\FF|}\label{markov-bad}\end{align}
When $|\SSS(W)|\leq \delta|\FF|$ we
define a new $d$-uniform multiset family $\FF'$ over $X\bs W$ which is 
$\mathbf s'=((1-\delta)|\FF|;s_1,\dots,s_d)$-spread.
The rest of the proof follows immediately from the 
arguments in Section 2.1 of~\cite{alwz}. 
\end{proof}

\begin{proof}[Proof of Theorem~\ref{intersecttheo}]
We roughly follow the argument used in \cite{alwz, rao}. Let $\FF$ be a $d$-intersecting, 
$n$-uniform family over $X$ of size $|\FF|> (4r)^n[C r\log(rd)]^d$, for $C$ to be chosen later. 
Let $T\subseteq X$ be the largest set with $|T|\leq d$ (possibly $|T|=0$) such that 
\[|\FF_T|\geq [C r \log (rd)]^{-|T|}|\FF|.\] 
We claim that $\FF_T$ is $\kappa =C r \log (rd)$-spread. Indeed, if $|T|=d$ then $\FF_T$
is a family of pairwise disjoint sets, and otherwise we can find a link at $T'\subseteq X\bs T$ such 
that $|T'|>0$ and
\[(\FF_T)_{T'}\geq [C r \log (rd)]^{-|T'|}|\FF_T|.\] 
But then, taking $T'\cup T$ gives a larger set with 
\[|\FF_{T\cup T'}|\geq [C r \log (rd)]^{-|T\cup T'|}|\FF|,\]
a contradiction. Hence,
$\FF_T$ is $\mathbf s$-spread and $(d-|T|)$-intersecting for weight profile 
\[{\mathbf s}=(|\FF_T|; |\FF_T|/\kappa,\dots,|\FF_T|/\kappa^{d-|T|},1,\dots, 1),\] taking 
$\sigma(F)=1$ for all $F\in\FF_T$. 
Let $\mathbf s'=(|\FF_T|/2;|\FF_T|/\kappa,\dots,|\FF_T|/\kappa^{d-|T|})$. 
As in~\cite{alwz}, we observe that if a family is 
$(|\FF_T|/2;|\FF_T|/\kappa,\dots,|\FF_T|/\kappa^{d-|T|})$-spread,
then it is also $\mathbf s''= (|\FF_T|;|\FF_T|/\kappa',\dots,|\FF_T|/\kappa'^{(d-|T|)})$-spread 
for $\kappa'=\kappa/2$. By Lemma~\ref{logn-spread-sat}, $\mathbf s''$ (and hence also 
$\mathbf s'$) is $(\frac{1}{2r},\frac{1}{2r})$-satisfying for $C$ chosen sufficiently large. 
Hence, by Lemma~\ref{good-size-d} with $\delta = 1/2$ and $p=\frac{1}{2r}$,
we know that $\mathbf s$ is $(1/r,1/r)$-satisfying choosing $C$ sufficiently large according
to the result of the lemma. Therefore, by Lemma~\ref{sat-disjoint},
this implies $\FF_T$ contains $r$ pairwise disjoint sets. Let 
$F_1,\dots,F_r\in\FF_T$ be pairwise disjoint. Then $F_1\cup T,\dots, F_r\cup T\in \FF $ 
is an $r$-sunflower.
\end{proof}


\section{Discussion}\label{discussion}
It is easy to see that Theorem~\ref{intersecttheo} implies a stronger statement than the best known 
bound whenever $d=o(n)$. The original bound of Erd\H{o}s and Rado can be directly
applied to $d$-intersecting sets families to achieve a bound of $(r-1)^{d+1}n!/(n-d)!$.
A natural question is when Theorem~\ref{intersecttheo} 
improves this trivial bound? There is some constant $c>0$ such that
\begin{align*} \frac{(r-1)^{d+1}n!}{(n-d)!} 
\geq \left[\frac{(r-1)n}{2}\right]^{c n /\log\log n}
\geq (4r)^n[Cr\log (rc n /\log\log n)]^{c n /\log\log n}
\end{align*} for $d\geq c n /\log\log n$ and $n$ sufficiently large. 
Hence, in this regime, we improve the best known bound. In particular, this motivates
Corollary~\ref{intcoro} and Question~\ref{lin-int-quest}.\\

{\noindent \bf Acknowledgement:} 
We thank Ryan Alweiss, Noga Alon, Robert Sedgewick and Stephen Melczer for helpful suggestions,
and to Lutz Warnke for telling us about~\cite{bcw}.
Thank you to The Fifty Five Fund for Senior Thesis 
Research (Class 1955) Fund, Princeton University for partially supporting this research.\\



\begin{thebibliography}{99}

\bibitem{alwz}
R. Alweiss, S. Lovett, K. Wu, J. Zhang, Improved bounds for the sunflower lemma, 
\textit{Ann. of Math.} (2) 194 (2021), no. 3, 795–815

\bibitem{bcw}
T. Bell, S. Chueluecha and L. Warnke, Note on sunflowers, \textit{Discrete Mathematics}, 344 (2021), no. 7, 112367

\bibitem{deza}
M. Deza, Solution d'un probl\`eme de Erd\H{o}s-Lov\'asz, \textit{J. Comb. Theory, Ser. B} 16 (2), 1974,
166–167.

\bibitem{df}
M. Deza and P. Frankl. Every large set of equidistant $(0,+1,-1)$-vectors forms a
sunflower. \textit{Combinatorica}, 1(3):225–231, 1981.

\bibitem{er}
P. Erd\H{o}s and R. Rado. Intersection theorems for systems of sets. \textit{Journal of the London
Mathematical Society}, 35(1):85-90, 1960.

\bibitem{gh}
G. Heged\H{u}s, Sunflowers and $L$-intersecting families, \textit{AKCE International Journal of 
Graphs and Combinatorics}, 17:1, 402-406, 2019.

\bibitem{rao}
A. Rao, Coding for Sunflowers, \textit{Discrete Analysis}, 2020:2.


\end{thebibliography}
\end{document}